\documentclass[12pt]{article}
\usepackage{graphicx}
\usepackage{amssymb}
\usepackage{amsthm}
\usepackage{amsmath}

\newcommand{\real}{{\mathbb R}}
\newcommand{\closure}{{\rm cl}}
\newcommand{\rme}{{\rm e}}
\newcommand{\rmi}{{\sqrt{-1}}}
\newcommand{\set}[1]{\left\{#1\right\}}
\newcommand{\zahl}{{\mathbb Z}}
\newcommand{\cx}{{\mathbb C}}
\newcommand{\rat}{{\mathbb Q}}

\newtheorem{corollary}{Corollary}
\newtheorem{definition}[corollary]{Definition}
\newtheorem{lemma}[corollary]{Lemma}

\newtheorem{theorem}[corollary]{Theorem}

\begin{document}

\title{Stripes on rectangular tilings}
\author{Akio Hizume \and Yoshikazu Yamagishi}
\maketitle

\begin{abstract}
We consider a class of cut-and-project sets $\Lambda = \Lambda_F
\times \zahl$ in the plane.  Let $L=\Lambda+w\real$, $w\in\real^2$, be
a countable union of parallel lines.  Then either (1) $L$ is a
discrete family of lines, (2) $L$ is a dense subset of $\real^2$, or
(3) each connected component of the closure of $L$ is homeomorphic to
$[0,1] \times \real$.


MSC 52C23, 11J72
\end{abstract}

\section{Introduction}
\label{intro}

One dimensional quasicrystal structure on the Penrose tiling
(\cite{debruijn}) was first recognized as the musical sequence of
Ammann bars \cite{grunbaum-shephard}.  Pleasants
\cite{pleasants2000,pleasants2003} studied the general theory of
submodels and quotient models of `plain' cut-and-project sets.  One
dimensional structure is also related to the tomography problem
\cite{huck}.

In \cite{jphysa}, we showed the following result.
\par\smallskip\noindent{\bf Theorem} Let $\Lambda_P\subset \cx$ be a
Penrose set.  Let $w\in\cx$, $w\neq0$.  If
$w\real\cap\zahl[\zeta]\neq0$, $\zeta=\rme^{2\pi\rmi/5}$, then the
quotient $(\Lambda_P+w\real)/(w\real)$ is a one-dimensional
cut-and-project set.  If $w\real\cap\zahl[\zeta]=0$, $\Lambda_P + w
\real$ is a dense subset of $\cx$.
\smallskip

In this paper, we consider a simpler family of model sets $\Lambda
\subset \real^2$ of the form $\Lambda = \Lambda_F \times \zahl$,
having a `rotation angle' parameter $\theta$ such that
$\tan\theta\in\real\setminus\rat$.  It is shown that $\Lambda$ has no
nontrivial one dimensional quasicrystal structure.  Instead, we show
the following.
\begin{theorem}
\label{1}
Let $w=(1,s)$, $s\neq0$. 
If $s \in \rat \cos\theta + \rat \sin\theta$, then either
\begin{enumerate}
\item $\Lambda+w\real$ is a dense subset of $\real^2$, or 
\item each connected component of the closure of
$\Lambda+w\real$ is homeomorphic to $[0,1] \times \real$.
\end{enumerate}
If $s\not\in \rat \cos\theta + \rat \sin\theta$, then 
$\Lambda+w\real$ is a dense subset of $\real^2$.
\end{theorem}

For the case of a special length of the window interval, we give
another proof by using the dynamics of a suspension map for an
interval exchange map.

\section{Stripes on rectangular tilings}

See \cite{moody} for a general discussion on the cut-and-project sets.
In this paper we adopt a simplified definition.
\begin{definition}
A cut-and-project scheme $\Sigma=(\real^k, \real, D, \Omega,
\Lambda, p, q)$ consists of a physical space $\real^k$, an internal
space $\real$, a lattice $D \subset \real^k \times \real$, a bounded
interval $\Omega \subset \real$ called the `window', a subset
$\Lambda=\Lambda(\Omega)\subset \real^k$, and natural
projections $p: \real^k \times \real \to \real^k$, $q:
\real^k \times \real \to \real$, such that
\begin{enumerate}
\item $p|D:D \to \real^k$ is injective, 
\item $q(D)$ is dense in $\real$, and
\item $\Lambda=\set{p(d) :\ d\in D, q(d)\in\Omega}$.
\end{enumerate}
A subset $\Lambda \subset \real^k$ is called a model set, or
cut-and-project set, if there exists a cut-and-project scheme
$\Sigma=(\real^k, \real, D, \Omega, \Lambda, p, q)$.  A relatively
dense subset $\Lambda \subset \real^k$ is called a Meyer set if it is
a subset of a model set.
\end{definition}

An example is the cut-and-project scheme
\[ \Sigma_F=(\real, \real, D_F, I, \Lambda_F, p, q)
\]
with a bounded interval $I=[0,\epsilon)$ and a lattice
\[ D_F
 = \set{(m \cos\theta - n \sin\theta, m \sin\theta + n \cos\theta)
   :\ m,n \in\zahl}
\]
where $\tan\theta$ is irrational.  The model set $\Lambda_F$ is written by
\[
 \Lambda_F = \set{ (m \cos\theta - n \sin\theta) :\ 
 0 \le m \sin\theta + n \cos\theta < \epsilon,
\ m,n\in\zahl}.
\]
If $\tan\theta=(\sqrt{5}-1)/2$, $\Lambda_F$ is called a Fibonacci set.
For $x=m\cos\theta-n\sin\theta$, $m,n\in\zahl$, we denote by $x^* =
m\sin\theta+n\cos\theta$.

In this paper we consider the cut-and-project scheme
\[ \Sigma = (\real^2, \real, D, I, \Lambda, p, q),
\]
with a bounded interval $I=[0,\epsilon)$ and a lattice
\[ D
 = \set{(m \cos\theta - n \sin\theta, k, m \sin\theta + n \cos\theta)
   :\ m,n,k\in\zahl}
\]
where $\tan\theta$ is irrational.  The model set $\Lambda
\subset \real^2$ is written by $\Lambda = \Lambda_F \times \zahl$.

For $w\in\real^2$, we denote by $\set{0}\times
S = (\Lambda +w\real) \cap (\set{0}\times\real)$.
\begin{theorem}
\label{2}
Let $w=(1,\lambda/d)$, $\lambda=a\cos\theta-b\sin\theta$, where
$a,b,d\in\zahl$ are relatively prime and $d\neq0$.  If $\epsilon \ge
1/|\lambda^*|$, $S$ is a dense subset of $\real$.  If $0 < \epsilon <
1/|\lambda^*|$, the closure of $S$ is written by
\begin{equation}
\label{3}
\closure(S) = \set{\frac{l + t \lambda^*}{d} :\  
 0 \le t \le \epsilon, \ l\in\zahl }.
\end{equation}
\end{theorem}
\begin{proof}
Let $x=m\cos\theta-n\sin\theta$.
The line passing through the point $(x,k) \in \Lambda$, with
the slope $\lambda/d$, passes through the point
$(0,k-\lambda x/d)$.
We have
\begin{align*}
k- \frac{\lambda x}{d}
&= k-\frac{a\cos\theta-b\sin\theta}{d} (m\cos\theta-n\sin\theta) \\
&= k-\frac{am+bn}{d}
  + \frac{a\sin\theta+b\cos\theta}{d} (m\sin\theta+n\cos\theta).
\end{align*}
It is not difficult to see that the set
\[ \set{(am+bn \bmod d,
  m\sin\theta+n\cos\theta) :\ m,n\in\zahl}
\]
is a dense subset of
$(\zahl/d\zahl)\times\real$, since $\tan\theta$ is irrational and
$a,b,d$ are relatively prime.  Thus we obtain (\ref{3}).
\end{proof}

The first part of Theorem~\ref{1} immediately follows from
Theorem~\ref{2}.

Here we give a picture.  Let $\theta=\pi/6$, $\epsilon =
\cos\theta+\sin\theta=(1+\sqrt{3})/2$.  Figure 1 draws $112$ lines of
the slope $\cos\theta+\sin\theta$ in the range
$[-2.5,2.5]\times[-2.5,2.5]$.  There appear thick `stripes'.
\begin{figure}
\includegraphics{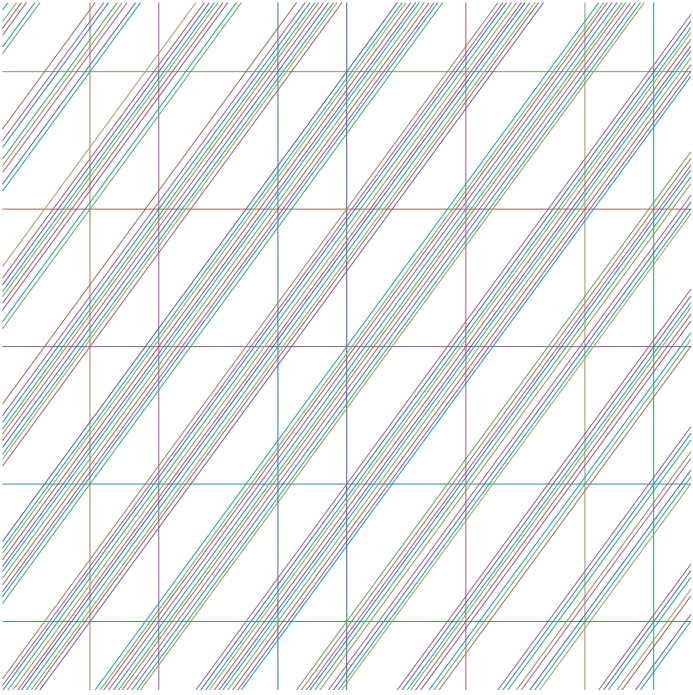}
\caption{Thick stripes.  $\theta=\pi/6$.  The slope is
  $\cos\theta+\sin\theta$.}
\label{fig:1} 
\end{figure}

\section{Suspension dynamics on the rectangular tiling}

In this section we prove (\ref{3}) for the case 
$\epsilon = \cos\theta + \sin\theta$, by using dynamical systems.
Since $\Lambda_F$ is a discrete subset of $\real$,
it is written by a sequence
\[ \Lambda_F = \set{ x_j = m_j \cos\theta - n_j \sin\theta:\ j\in\zahl},
\]
where $x_0=0$, and $x_j < x_{j+1}$, $j\in\zahl$.

\begin{lemma}
If $x_j^* \in [0,\cos\theta)$, 
$x_{j+1}- x_j = \cos\theta$ and 
$x_{j+1}^*- x_j^* = \sin\theta$.
If $x_j^* \in [\cos\theta, \cos\theta+\sin\theta)$,
$x_{j+1}- x_j = \sin\theta$ and 
$x_{j+1}^*- x_j^* = -\cos\theta$.
\end{lemma}
\begin{proof}
The proof is trivial.
\end{proof}

Let $w=(1,\lambda/d)$, $\lambda=a\cos\theta-b\sin\theta$, where
$a,b,d\in\zahl$ are relatively prime.
Consider a line
$\ell_\alpha = \set{(t,\alpha+t\lambda/d) :\ t\in\real}$ 
passing through the point $(0,\alpha)$, with the slope $\lambda/d$.
Let
\[ X = [0,\cos\theta) \times [0, \frac{1}{d}] \times [0,\cos\theta]
 \cup [\cos\theta,\cos\theta+\sin\theta)
 \times [0,\frac{1}{d}] \times [0,\sin\theta]
\]
be the phase space of suspension dynamics, where we identify
\[
  (\xi,0,\zeta) \sim (\xi, \frac{1}{d}, \zeta),
\]
\[
  (\xi, \eta,\cos\theta) \sim (\xi+\sin\theta, \eta, 0),
 \quad 
 0\le \xi < \cos\theta,
\]
\[
  (\xi,\eta,\sin\theta) \sim (\xi-\cos\theta, \eta, 0), 
 \quad
  \cos\theta \le \xi < \cos\theta+\sin\theta.
\]
Let $\varphi : \real^2 \to X$ be a `card album' map defined by
\[ \varphi(x,y) = (x_j^*, y \bmod 1, x-x_j), 
 \quad \text{ for } \ x_j \le x < x_{j+1}.
\]
\begin{lemma}
 $\varphi(\Lambda) = \set{(x_j^*,0,0) :\ j\in\zahl}$.
 The mapping $\varphi : \real^2 \to X$ is continuous.
\end{lemma}
\begin{proof}
The proof is trivial.
\end{proof}

Let 
\[ M = [0,\cos\theta+\sin\theta) \times [0,\frac{1}{d}]
\]
be the base space of $X$, where we identify $(\xi,0) \sim (\xi,1/d)$.
$M$ is identified with $M\times\set{0} \subset X$.
We consider the `Poincar\'e map' in the trajectory 
$\varphi(\ell_\alpha)$ as follows.

\begin{lemma}
$\varphi(\ell_\alpha) \cap M$ is a dense subset of the `line'
\[
 L_\alpha := \set{(\xi,\alpha - \frac{\lambda^*}{d} \xi \bmod \frac{1}{d})
  :\ 0 \le \xi <\cos\theta+\sin\theta}.
\]
\end{lemma}
\begin{proof}
First we have
\[ \varphi(\ell_\alpha) \cap M
 = \set{(x_j^*, \alpha + \frac{\lambda x_j}{d} \bmod \frac{1}{d})
  :\ j\in\zahl},
\]
and $(x_0^*, \alpha + \frac{\lambda x_0}{d} \bmod \frac{1}{d} )
 =(0,\alpha \bmod \frac{1}{d})$.
The mapping of $\varphi(\ell_\alpha) \cap M$, defined by
\begin{align*}
 & (x_{j+1}^*, \alpha + \frac{\lambda x_{j+1}}{d} \bmod \frac{1}{d}) \\
 & = \begin{cases}
  (x_j^*+\sin\theta,
  \alpha + \frac{\lambda}{d} (x_{j}+\cos\theta) \bmod \frac{1}{d})
    & \text{ if } 0 \le x_j^* < \cos\theta \\
  (x_j^*-\cos\theta, 
 \alpha + \frac{\lambda}{d} (x_{j}+\sin\theta) \bmod \frac{1}{d})
    & \text{ if } \cos\theta \le x_j^* < \cos\theta+\sin\theta,
 \end{cases}
\end{align*}
extends to a piecewise continuous map on $M$,
which could be called an `interval exchange' map.
Since
\[
  \frac{1}{\sin\theta}
  (\frac{\lambda \cos\theta}{d} - \frac{a}{d} )
  =  \frac{1}{-\cos\theta}
  (\frac{\lambda \sin\theta}{d} + \frac{b}{d} )
 = - \frac{\lambda^*}{d},
\]
the slope of $L_\alpha$ is well-defined.  The set $\set{x_j^* :\ j\in\zahl}$
is an orbit of an `irrational rotation' on the `circle'
$[0,\cos\theta+\sin\theta)$, so $\varphi(\ell_\alpha) \cap M$ is a
  dense subset of $L_\alpha$.
\end{proof}

\begin{lemma}
\label{4}
The line $L_\alpha$ passes through the point $(x_j^*,0)$ if and only if
$\alpha \equiv \frac{\lambda^* x_j^*}{d} \bmod \frac{1}{d}$.
\end{lemma}
\begin{proof}
The proof is trivial.
\end{proof}

Theorem~\ref{1} for the case $\epsilon = \cos\theta + \sin\theta$
immediately follows from Lemma~\ref{4}.

\section{Dense lines in the plane}

In this section we prove the second part of Theorem~\ref{1}.
\begin{theorem}
Let $w=(1,s)$.  If $s \not\in \rat \cos\theta + \rat \sin\theta$, then
$\Lambda+w\real$ is a dense subset of $\real^2$.
\end{theorem}
\begin{proof}
The lattice $D\subset\real^3$ is written by 
$D=AD_0$,
where 
$D_0=\zahl^3$ and $A=\begin{pmatrix}
 \cos\theta & 0 & -\sin\theta \\
 0 & 1 & 0 \\
 \sin\theta & 0 & \cos\theta
 \end{pmatrix}$.
Let $H=\real^2\times\set{0}$.  By giving a parallel translation to the
lattice $D_0$, we may assume that the window interval is
$I=[-\epsilon_1,\epsilon_2)$, $\epsilon_1,\epsilon_2>0$,
  $\epsilon_1+\epsilon_2=\epsilon$, so that $H\times I$ is a
  neighborhood of $H$.

We identify $\Lambda$ with $\tilde\Lambda=\Lambda\times\set{0} \subset
\real^3$, and rotate it back to $A^{-1}\tilde\Lambda$ in the slanted
plane $A^{-1}H$.  The slope vector $\tilde{w}=(1,s,0)$ is slanted to
$A^{-1}\tilde{w} = (\cos\theta,s,-\sin\theta)$.

Here we recall Kronecker's Theorem.

\smallskip
\noindent{\bf Kronecker's Theorem.}\ 
Let $\theta_1,\theta_2,\theta_3 \in \real$.
Suppose that no linear relation
$$a_1 \theta_1 + a_2 \theta_2 + a_3\theta_3=0$$
with integral coefficients, not all zero, holds between them.
Then
\[ \set{(t\theta_1 \bmod 1, t\theta_2 \bmod 1, t\theta_3 \bmod 1)
 :\ t\in\real}
\]
is a dense subset of the unit cube $[0,1)^3$.
\smallskip

Kronecker's Theorem implies that the line $A^{-1}\tilde\ell :=
\set{A^{-1}(\tilde{p}+t\tilde{w}) :\ t\in\real}$, for any $\tilde{p}
\in H$, comes arbitrarily close to the cutted lattice $D_0 \cap
A^{-1}(H\times I)$.  Hence $A^{-1}\tilde\ell$ comes arbitrarily close
to the slanted model set $A^{-1}\tilde\Lambda$.  This completes the
proof.
\end{proof}

Figure 2 draws $241$ lines in the same range as in Figure 1, of the
slope $\sqrt{2}$.  This is a finite approximation of the dense set in
the plane.
\begin{figure}
\includegraphics{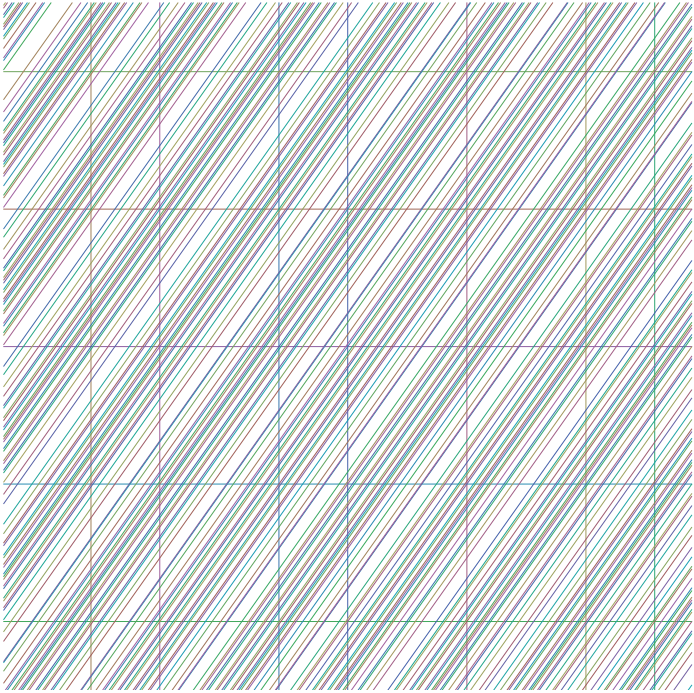}
\caption{Dense lines in the plane.  The slope is $\sqrt{2}$.}
\label{fig:2}
\end{figure}

\end{document}